\documentclass[12pt,a4paper]{amsart}
\usepackage{amssymb}
\usepackage{mathrsfs}
\usepackage{amsfonts}
\usepackage[active]{srcltx}
\usepackage{enumerate}
\usepackage{color}

\usepackage{amsmath,amssymb,xspace,amsthm}

\newtheorem{theorem}{Theorem}[section]
\newtheorem{lemma}[theorem]{Lemma}%[section]
\newtheorem{proposition}[theorem]{Proposition}%[section]
\newtheorem{corollary}[theorem]{Corollary}%[section]
\numberwithin{equation}{section}

\def\Der{\operatorname{Der}}
\def\span{\operatorname{span}}

\newcommand{\spn}{\operatorname{span}}

\newcommand{\ot}{\otimes}
\newcommand{\ad}{\operatorname{ad}\xspace}

\newcommand{\Rad}{\operatorname{Rad}\xspace}

\renewcommand{\a}{\ensuremath{\alpha}}

\renewcommand{\l}{\ensuremath{\lambda}}

\renewcommand{\leq}{\leqslant}
\renewcommand{\geq}{\geqslant}
\newcommand{\e}{\epsilon}
\newcommand{\p}{\partial}
\def\o{\omega}

\def\sl{\mathfrak{sl}}
\def\gl{\mathfrak{gl}}

\newcommand{\fh}{\ensuremath{\mathfrak{h}}}

\newcommand{\Z}{\ensuremath{\mathbb{Z}}\xspace}
\newcommand{\N}{\ensuremath{\mathbb{N}}\xspace}

\newcommand{\W}{\ensuremath{\mathcal{W}}\xspace}

\newcommand{\C}{\ensuremath{\mathbb C}\xspace}

\def\LL{\mathcal{L}}
\def\hLL{\widehat{\mathcal{L}}}
\def\GG{\mathcal{G}}

\def\HH{\mathcal{H}}

\newcommand{\ba}{{\bf a}}
\newcommand{\bl}{{\bf l}}
\newcommand{\bm}{{\bf m}}
\newcommand{\bn}{{\bf n}}
\newcommand{\bt}{{\bf t}}
\newcommand{\br}{{\bf r}}
\newcommand{\bs}{{\bf s}}
\newcommand{\bi}{{\bf i}}
\newcommand{\bj}{{\bf j}}
\newcommand{\bk}{{\bf k}}

\newcommand{\bu}{{\bf u}}
\newcommand{\bv}{{\bf v}}
\newcommand{\bw}{{\bf w}}
\begin{document}
\title[Irreducible modules over divergence zero algebras and $q$-analogues]
{Irreducible modules over the divergence zero algebras and their $q$-analogues}
\author{Xuewen Liu, Xiangqian Guo and Zhen Wei}
\date{} \maketitle
\begin{abstract}
In this paper, we study a class of $\Z_d$-graded modules, 
which are constructed using Larsson's functor from $\sl_d$-modules $V$,    
%with infinite-dimensional weight spaces 
for the Lie algebras of divergence zero vector fields on tori and quantum tori. 
We determine the irreducibility of these modules for finite-dimensional or infinite-dimensional $V$ using a unified method. In particular, these modules provide new irreducible weight modules with infinite-dimensional weight spaces for the corresponding algebras.
%
%In this paper, we construct a class of weigh modules with
%infinite-dimensional weight spaces for the Lie algebras of
%divergence zero vector fields on tori and Lie algebras closely
%related to them. We determine the irreducibility of these modules,
%using modules of the general linear Lie algebras. We also obtain similar
%result for the $q$-analogues of these Lie algebras. 
%
%
%In\cite{E}, Rao determined the necessary and sufficient condition
%for the weight modules $F_b^\a(V)$ over the Witt algebra $\W_d$ to
%be irreducible when $V$ is finite dimensional, then Liu considered
%the irreducibility  when $V$ is infinite dimensional follow on in
%\cite{LZ}. In this note, we will generalize Liu's results to the
%subalgebra $\widehat{\LL}_d$ and $\LL_d$ of Witt algebra $\W_d$,
%which are the Lie algebra of divergence zero vector fields.
%Therefore we obtain a lot of irreducible infinite dimensional
%$\LL_d$-modules with infinite weight spaces.
%When $d=2$, the corresponding algebras are Virasoro-like algebra and
%a related Lie algebra. So we can obtain a class of irreducible
%weight modules with infinite-dimensional weight spaces for the
%Virasoro-like algebra.
\end{abstract}

\vskip 5pt \noindent{\em Keywords:} Divergence zero algebras, Witt algebras, quantum tori, irreducible modules. 
\vskip 5pt

\section{Introduction}\label{introduction}

Representation theory for infinite-dimensional Lie algebras have
been attracting extensive attentions of many mathematicians and physicists. 
These algebras include Kac-Moody algebra, (generalized) Virasoro algebra, 
Witt algebras, Cartan type Lie algebras, Lie algebras of Block type and so on. 
For any positive integer $d$, the Witt algebra $\W_d$ is the
derivation algebra of the Laurent polynomial algebra $A_d=\C[t_1^{\pm1},\ldots,t_d^{\pm 1}]$. 
This algebra is also known as the Lie algebra of the group of diffeomorphisms of the $d$-dimensional torus. 
%The algebra $\W_d$ is a natural higher rank generalization of the
%centerless Virasoro algebra, which has wide application in various
%branches of mathematics and physics (see \cite{M, KR}).

The representation theory of Witt algebras was
studied extensively by many authors, see \cite{B, BF, BMZ, E1, E2, GZ, LZ, T} and
references therein. 
In 1986, Shen \cite{Sh} defined a class of modules
$F_b^\a(V)$ over the Witt algebra $\W_d$ for any weight modules over
the special linear Lie algebra $\sl_d$, where $\a$ is any $d$-dimensional complex vector 
and $b$ is a complex number.
These modules as well as the
functors $F^\a_b$, known as the Larson functors now, were also studied by
Larson in \cite{L1,L2,L3}. In 1996, Eswara Rao \cite{E1} determined the
irreducibility of these modules for finite-dimensional $V$
%, see laso \cite{GZ}
, and recently G. Liu and K. Zhao studied the case when
$V$ is infinite-dimensional (see \cite{LZ}).

Geometrically, $\W_d$ may be interpreted as the Lie algebra of
(complex-valued) polynomial vector fields on a $d$-dimensional
torus. It has a natural subalgebra consisting of vector fields of
divergence zero, which we denote by $\hLL_d$, and another subalgebra $\LL_d$, 
which is obtained from $\hLL_d$ modulo the Cartan subalgebra of $\hLL_d$. 
They are also known as the Cartan S type Lie algebras. The modules $F^\a_b(V)$
admit natural module structures for the algebra $\hLL_d$
and $\LL_d$, and the $\hLL_d$- or $\LL_d$-module structures on $F^\a_b(V)$
do not depend on the parameter $b$, so we will denote them by $F^\a(V)$. 
Recently, Talboom \cite{T} determined the irreducibility
of the $\hLL_d$-module $F^\a(V)$ for finite-dimensional
$V$ and Billig and Talboom \cite{BT} investigated the category $\mathcal{J}$ of jet modules for $\hLL_d$.
%
%When $d=2$, the algebra $\LL_d$ is just the centerless Virasoro-like
%algebra, which is one of the infinite-dimensional simple Lie
%algebras with simplest structures and has many similarities with the
%Virasoro algebra both in structure theory and in representation
%theory. The representation theory of the Virasoro-like algebra was
%also investigated extensively (See \cite{LT, GL, GLi, JL}) and reference
%therein.

Let $q=(q_{ij})_{i,j=1}^d$ be a $d\times d$ matrix over $\C$, 
where $q_{ij}=q_{ji}^{-1}$ are roots of unity. 
We have the $d$-dimensional rational quantum torus $\C_q=\C_q[t_1^{\pm1},\cdots, t_d^{\pm1}]$ (see \cite{N}),
which has been used to characterize the extended affine Lie algebras in \cite{AABGP}.
The derivation Lie algebra $\Der(\C_q)$ is a $q$-analogue of the Witt algebra $\W_d$. 
The representation theory of $\Der(\C_q)$ has been studied by many mathematicians (\cite{LT1,LZ1.5,MZ,Z}).
The Lie algebra $\Der(\C_q)$ has a natural subalgebra $\hLL_d(q)$, called the skew
derivation Lie algebra of $\C_q$. 
Removing the Cartan subalgebra from $\hLL_d(q)$, we obtain another natural subalgebra
$\LL_d(q)$. The algebras $\hLL_d(q)$ and $\LL_d(q)$ are the $q$-analogues of the algebras
$\hLL_d$ and $\LL_d$ introduced in the previous paragraph.
Similarly, for any $\sl_d$-module $V$, we can construct the Shen-Larsson module for $\hLL_d(q)$ and $\LL_d(q)$,
which we denote by $F_q^\a(V)$. 
The structures of the modules $F_q^\a(V)$ are studied by \cite{LT2} for the case $d=2$ and $V$ being finite-dimensional $\sl_2$-modules.

In the present paper, we will determine the irreducibility of
$F^\a(V)$ as modules over the divergence zero algebra
 $\hLL_d$ as well as $\LL_d$ for arbitrary $\sl_d$-module $V$ in a unified way. 
Then we generalize our results to the $q$-analogue algebras $\hLL_d(q)$ and $\LL_d(q)$.
The paper is organized as follows. 
In Section 2, we first introduce the
notation of the Witt algebra $\W_d$, its subalgebras $\hLL_d$ and
$\LL_d$, then we construct the modules
$F^{\a}(V)$ from irreducible special linear Lie algebra modules $V$. 
In Section 3, we will determine the irreducibility of $F^\a(V)$
as modules over $\hLL_d$ or $\LL_d$ for irreducible $\sl_d$-modules $V$.
In fact, we can give a description of all their submodules.
We will handle the problems for both finite-dimensional and
infinite-dimensional $V$ using a unified method.
When $V$ is finite-dimensional, our results for $\hLL_d$ 
recover the main result of Talboom \cite{T} and our results for $\LL_d$ are new.  
When $V$ is infinite-dimensional, our results provide new simple weight modules
with infinite-dimensional weight spaces for both $\hLL_d$ and $\LL_d$.
In the last section, we will generalize our results to the $q$-analogue algebras $\hLL(q)_d$ 
and $\LL(q)_d$, which are natural subalgebras of derivation Lie algebra of quantum tori, where $q$ is a root of unity. 
In particular, we get new simple weight modules with infinite-dimensional weight spaces
for these algebras.

%Actually we prove that, for any infinite
%dimensional irreducible $\mathfrak{sl}_d$-modules $V$, any
%$\a\in\C^d$ and $b\in\C$, the $\LL_d$-modules
%$F_b^{\a}(V)=V\otimes A_d$ is always irreducible, see Theorem
%\ref{infinite_d}. Since $\LL_d$ is a subalgebra of $\hLL_d$,
%$F^\a(V)$ is automatically irreducible over $\hLL_d$. In
%particular, when $d=2$, we obtain a class of irreducible weight
%modules with infinite-dimensional weight spaces for the
%Virasoro-like algebra.

Our results for $\hLL_d$ and $\LL_d$ generalize the similar result of Liu and Zhao for
the Witt algebra $\W_d$ in \cite{LZ} and the result of Talboom for
the algebra $\hLL_d$ for finite-dimensional $V$ (\cite{T}). 
Our results for the algebras $\hLL_d(q)$ and $\LL_d(q)$ are new except for the case 
when $d=2$ and $V$ is finite-dimensional.
The main difficulty in our question is that, unlike the algebras $\W_d$ or
$\hLL_d$, the algebras $\LL_d$ and $\LL_d(q)$ do not admit Cartan
subalgebras and hence any submodule of a weight module may not be a
weight module automatically.

\section{Preliminaries}

We denote by  $\C,\Z,\Z_+,\N$  the sets of all complex numbers, all
integers, all non-negative integer and all positive integer,
respectively. All vector spaces and algebras are over $\C$.

Fix a positive integer $d\geq 2$. For any
$\bn=(n_1,\cdots,n_d)^T\in\Z_{+}^d$ and $\ba=(a_1,\cdots,a_d)^T\in\C^d$,
we denote $\ba^\bn=a_1^{n_1}a_2^{n_2}\cdots a_d^{n_d}$, where $T$ means
taking transpose of the matrix. Let $\mathfrak{gl}_d$ be the Lie
algebra of all $d\times d$ complex matrices, and $\mathfrak{sl}_d$ be
the subalgebra of $\mathfrak{gl}_d$ consisting of all traceless
matrices.

Let $A_d=\c[t_1^{\pm1},t_2^{\pm1},\cdots,t_d^{\pm 1}]$ be the
algebra of Laurent polynomials over $\C$ and denote by $\W_d$ the
Lie algebra of all derivations of $A_d$, called the Witt algebra.
Set $\partial_i=t_i\frac{\partial}{\partial t_i}$ for
$i=1,2,\cdots,d$, and  $\bt^\bn=t_1^{n_1}t_2^{n_2}\cdots t_d^{n_d}$ for
$\bn=(n_1,n_2,\cdots,n_d)^T\in\Z^d$.

Let $(\cdot|\cdot)$ be the standard symmetric bilinear form on $\C^d$, that is,
$(\bu\ |\ \bv)={\bu}^T{\bv}\in\C$. Homogeneous elements of $\W_d$ with respect
to the power of $t$ will be denoted
$D(\bu, \br)=\bt^{\br}\sum_{i=1}^du_i\partial _i$ for $\bu=(u_1,\cdots,u_d)^T\in\C^d, \br\in\Z^d$.
Then $\W_d$ is spanned by all $D(\bu, \br)$ with $\bu\in\C^d$ and
$\br\in\Z^d$. The Lie bracket of $\W_d$ is given by
\begin{equation*}
[D(\bu, \br), D(\bv, \bs)]=D(\bw, \br+\bs), \bu, \bv\in\C^d, \br, \bs\in\Z^d,
\end{equation*}
where $\bw=(\bu\ |\ \bs)\bv-(\bv\ |\ \br)\bu$. Geometrically, $\W_d$ may be interpreted
as the Lie algebra of (complex-valued) polynomial vector fields on a
$d$-dimensional torus. Then we can deduce a subalgebra of $\W_d$,
the Lie algebra of divergence zero vector fields, denoted by $\hLL_d$. It is spanned by
$D(\bu, \br)$ which satisfy $(\bu\ |\ \br)=0$, that is,
$$\hLL_d=\spn_\C\{\p_i,\ \bt^\br (r_j\p_i - r_i\p_j)\ |\ i, j=1,\cdots,d, \br\in\Z^d\}.$$
Note that $\hLL_d$ has the Cartan subalgebra
$$\HH=\spn_{\C}\{\p_j\ |\ j=1,\cdots,d\}=\{D(\bu,0)\ |\
\bu\in\C^d\}.$$ The algebra $\hLL_d$ has a natural
subalgebra $$\LL_d=\spn_\C\{\bt^\br (r_j \p_i - r_i\p_j)\ |\ i,
j=1,\cdots,d, \br\in\Z^d\}.$$ We see that the algebra $\LL_d$ does
not admit a Cartan subalgebra.
When $d=2$, the algebra $\LL_d$ is just the centerless Virasoro-like
algebra.

Now we define some modules for these algebras. For any $\alpha\in\C^d,b\in\C$,
and $\mathfrak{gl}_d$-module $V$ on which the identity matrix acts as the
scalar $b$, let $F_b^{\a}(V)=V\otimes A_d$. 
%For convenience, we write $v(\bn)=v\otimes \bt^\bn$ for any $v\in V, \bn\in \Z^d$. 
Then $F_b^{\a}(V)$
becomes a $\W_d$-module with the following actions
\begin{equation*}
D(\bu, \br)(v\ot \bt^{\bn})=((\bu\ |\ \bn+\a)v+(\br\bu^T)v)\ot \bt^{\bn+\br},
\end{equation*}where $\bu\in\C^d, v\in V, \bn, \br\in\Z^d$.

When $V$ is a finite-dimensional $\sl_d$-module, $F_b^{\a}(V)$ has
all weight spaces finite-dimensional, and the irreducibility of
these modules are determined by \cite{E1} (see also \cite{GZ}). It
was conjectured that all irreducible $\W_d$-modules with finite
dimensional weight spaces are precisely the generalized highest
weight modules and the irreducible sub-quotient modules of $F^\a(V)$
for finite-dimensional irreducible $\sl_d$-modules $V$. This
conjecture was recently proved by \cite{BF}. When $V$ is
infinite-dimensional, it was shown that $F_b^{\a}(V)$ is always
irreducible by \cite{LZ}.

Since $\widehat{\LL}_d$ and $\LL_d$ are subalgebras of $\W_d$, by
restriction of the module action, $F_b^{\a}(V)$ can be viewed as a
module for the algebras $\widehat{\LL}_d$ and $\LL_d$. This module
are not dependent on $b$, so we denote the resulted
$\widehat{\LL}_d$-module or $\LL_d$-module by $F^\a(V)$. In
(\cite{T}), Talboom determined the irreducibility of
$F^\a(V)$ as a module over $\hLL_d$ when $V$ is finite-dimensional.
In this paper, we will determine the irreducibility of $F^\a(V)$ as
a module over the algebra $\hLL_d$ as well as the algebra $\LL_d$, when
$V$ is an arbitrary simple $\sl_d$-module. The main difficulty is that the algebra
$\LL_d$ does not admit a Cartan subalgebra. 
%Our methods also works for
%finite-dimensional and infinite-dimensional $\sl_d$-module $V$.

\section{Modules over the Divergence Zero Algebra}

%Now we generalize our previous result for $\GG$ to the general case.
Let notation as in the last section. In particular, let $V$ be an
irreducible $\sl_d$-module, $\a\in\C^d$ and $\LL_d$ is the
divergence zero algebra and we have constructed the $\LL_d$-module
$F^\a(V)$. Then we recall some known results on finite-dimensional
modules over $\sl_d$.

Fix a standard basis of $\sl_d$: $\{E_{i,j}, E_{k,k}-E_{k+1,k+1},\
|\ i,j=1,\cdots, d, k=1,\cdots, d-1, i\neq j\},$ where $E_{i,j}$ is
the matrix with $1$ as the $(i,j)$-entry and $0$ otherwise. Let
$\{\mu_1,\cdots,\mu_{d-1}\}$ be the coordinate functions on the
$d\times d$ diagonal matrices, i.e., $\mu_i(E_{j,j})=\delta_{i,j}$.
Then $\fh=\span\{E_{k,k}-E_{k+1,k+1}, k=1,2,\cdots, d-1\}$ is just
the Cartan subalgebra, $\mu_i-\mu_{i+1}$ is a base for the standard
root system of $\sl_d$ and $\omega_{k}=\mu_1+\cdots+\mu_{k},
k=1,\cdots,d-1$ are the fundamental dominant weights. Then any
dominant weight is a $\Z_+$-linear combination of these $\omega_k$
and any finite-dimensional irreducible $\sl_d$-module is a highest
weight modules $V(\l)$ with a dominant highest weight $\l$.
For convenience, we also denote $\omega_{d}=\mu_1+\cdots+\mu_{d}=0$.

The following technical lemma on $\sl_d$-modules are taken from
\cite{LZ}.

\begin{lemma}\label{sl_d-mod}
Let $V$ be an irreducible $\sl_d$-module (not necessarily weight module).
\begin{itemize}
\item [(1)]For any $i, j$ with $1\leq i\neq j\leq d$, $E_{ij}$ acts injectively or
   locally nilpotently on $V$.
\item [(2)] The module $V$ is finite-dimensional if and only if $E_{ij}$ acts locally
   nilpotently on $V$ for any $i, j$ with $1\leq i\neq j\leq d$.
\end{itemize}
\end{lemma}

For convenience, we
denote $d_{\br,i}=\bt^{\br}(r_{i+1}\partial_i-r_i\partial_{i+1})$ for
$\br=(r_1,\cdots,r_{d})^T\in\Z^d$ and $i=1,\cdots,d-1$. Let $V$ be an irreducible $\sl_d$-module and 
$F^\a(V)$ the corresponding $\LL_d$-module.
The action of $\LL_d$ on $F^\a(V)$ can be rewritten as
\begin{equation*}
d_{\br,i}v\ot\bt^{\bn}=\big((r_{i+1}\e_i-r_i\e_{i+1} | \bn+\a)+\br(r_{i+1}\e_i-r_i\e_{i+1})^T\big)v\ot\bt^{\bn+\br},\ i=1,\cdots,d-1,
\end{equation*}
where $v\in V, \bn, \br\in\Z^d$, $\e_i\in\Z^d$ is the
standard basis vector with $1$ as the $i$-th entry and $0$
otherwise. Then we have

\begin{theorem}\label{element_d} Let $V$ be an irreducible $\sl_d$-module such that $V\not\cong V(\omega_k)$ for any $k=1,\cdots,d$ and $N$ an $\LL_d$-submodule of $F^\a(V)$. Then
there exists nonzero $v\in V$ such that $v\ot \bt^\bn\in N$ for all
$\bn\in\Z^d$.
\end{theorem}

\begin{proof}
Let $N$ be a nonzero $\LL_d$-submodule of $F^\a(V)$. % Next we will show $W=V\otimes A_d$.
Choose a nonzero vector $\sum_{\br\in I}v_{\br}\otimes \bt^{\br}\in
N$, where $I\subseteq \Z^d$ is a finite subset. For any
$\bm=(m_1,\cdots,m_d)^T, \bn=(n_1,\cdots,n_d)^T\in\Z^d$ and $i,j\in
\{1,2,\cdots,d-1\}$, we have
\begin{equation}\label{d:m-n,n}\begin{split}
  &  d_{\bm-\bn,i}d_{\bn,j}(\sum_{\br\in I}v_{\br}\otimes \bt^{\br})\\
= & \big((\bu_1 | \a+\bn+\br)+(\bm-\bn)\bu_1^T\big)\big((\bu_2\ |\ \a+\br)+\bn \bu_2^T\big)({\sum_{\br\in I}v_{\br}\otimes \bt^{\bm+\br}})\\
= & \Big((\bu_1 | \a+\bn+\br)+\sum_k(m_k-n_k)(m_{i+1}-n_{i+1})E_{k, i} -\sum_k(m_k-n_k)(m_{i}-n_{i})E_{k, i+1}\Big)\\
  & \hskip5pt \cdot\Big((\bu_2 | \a+\br)+\sum_ln_ln_{j+1}E_{l,j}-\sum_ln_ln_{j}E_{l,j+1}\Big)\sum_{\br\in I}v_{\br}\otimes \bt^{\bm+\br}\in N.\\
\end{split}\end{equation}
where $\bu_1=(m_{i+1}-n_{i+1})\e_{i}-(m_{i}-n_{i})\e_{i+1},
\bu_2=n_{j+1}\e_{j}-n_{j}\e_{j+1}\in\C^d$. Consider the right-hand side of \eqref{d:m-n,n} as a
polynomial in $\bn$ with coefficients in $N$, and using the
Vandermonde's determinant, we can deduce that,
\begin{equation}\label{polynomial}\begin{split}
  & \sum_{\br\in I}\sum_{k,l=1}^d\Big(n_kn_{i+1}E_{k, i} -n_kn_{i}E_{k, i+1}\Big)\Big(n_ln_{j+1}E_{l,j}-n_ln_{j}E_{l,j+1}\Big)v_{\br}\otimes \bt^{\bm+\br}\\
= & \sum_{\br\in I}\sum_{k,l=1}^d\Big(n_kn_{i+1}n_ln_{j+1}E_{k,i}E_{l,j}+n_kn_{i}n_ln_{j}E_{k,i+1}E_{l,j+1}\\
  & \hskip10pt -n_kn_{i}n_ln_{j+1}E_{k,i+1}E_{l,j}-n_kn_{i+1}n_ln_{j}E_{k,i}E_{l,j+1}\Big)v_{\br}\otimes \bt^{\bm+\br}\in N,\ \forall\ \bm\in\Z^d.\\
\end{split}\end{equation}

\textbf{Step 1.} If $V$ is infinite-dimensional, then by Lemma
\ref{sl_d-mod}, the action of $E_{st}$ on $V$ is injective  for some
$1\leq s\neq t\leq d$. Without loss of generality, we may assume
that $s>t$. Take $i=j=t$ in \eqref{polynomial}, then considering the
coefficient of $n_s^2n_{t+1}^2$ gives $\sum_{\br\in
I}E_{st}E_{st}v_{\br}\otimes \bt^{\bm+\br}\in N$ for all $\bm\in\Z^d$,
thanks to the Vandermonde's determinant again. By replacing
$v_{\br}$ with $E_{st}E_{st}v_{\br}$ we may assume that
\begin{equation}\label{d step-1}
\sum_{\br\in I}v_{\br}\otimes \bt^{\bm+\br}\in N,\ \forall\
\bm\in\Z^d.
\end{equation}

Now we suppose that $V$ is finite-dimensional and hence a highest
weigh module $V(\l)$, with $\l$ being a dominant weight. Take any
$\br_0\in I$ such that $v_{\br_0}=v_1+\cdots +v_l\neq 0$ with $v_i$
being weight vectors of different weights and the weight of $v_1$ is
maximal. If the weight of $v_1$ is not $\l$, then there exists
$1\leq j\leq d-1$ such that $E_{j,j+1}v_1\neq 0$ has larger weight.
Now we consider that
\begin{equation}\label{d:n}\begin{split}
    & d_{\bn,j}\sum_{\br\in I}v_{\br}\otimes \bt^{\br}\\
%  = & D(n_{j+1}\e_j-n_j\epsilon_{j+1} | \bn)\sum_{\br\in I}v_{\br}\otimes \bt^{\br}\\
  = & \big((n_{j+1}\e_j-n_j\epsilon_{j+1} | \a+\br)+\bn (n_{j+1}\e_j-n_j\epsilon_{j+1})^T\big) {\sum_{\br\in I}v_{\br}\otimes \bt^{\bn+\br}}\\
  = & \sum_{\br\in I}\Big((n_{j+1}\e_j-n_j\epsilon_{j+1} | \a+\br)+ \sum_{k=1}^d(n_kn_{j+1}E_{k,j}-n_kn_jE_{k,j+1})\Big)v_{\br}\otimes \bt^{\bn+\br}\in N.
\end{split}\end{equation}
Take suitable $\bn\in\Z^d$ (for example, making $n_j\gg n_i$ for
$i\neq j$), we see that $-n_j^2E_{j,j+1}v_1$ is a nonzero weight
component of $d_{\bn,j}v_{\br_0}$ with maximal weight larger than
that of $v_1$. Replacing $\sum_{\br\in T}v_\br\ot \bt^{\br}$ with $d_{\bn,j}\sum_{\br\in T}v_\br\ot \bt^{\br}$, and repeating the above process several times, we may
assume that the weight of $v_1$ is $\l$.

Suppose that $\l=\sum_{k=1}^{d-1}a_k\omega_k$ for some $a_k\in\Z_+$.
Since $\l\neq w_k$ for any $k=1,\cdots,d$, there exist some $1\leq
k_1\leq k_2\leq d-1$ such that $a_{k_1}+a_{k_2}\geq2$. Let
$\mathfrak{s}$ be the $3$-dimensional simple Lie algebra spanned by
$E_{k_1,k_2+1}, E_{k_2+1,k_1}$ and $E_{k_1,k_1}-E_{k_2+1,k_2+1}$. Note that
$\l(E_{k_1,k_1}-E_{k_2+1,k_2+1})=a_{k_1}+a_{k_1+1}+\cdots+a_{k_2}$, hence the
$\mathfrak{s}$-module generated by $v_1$ has highest weight
$a_{k_1}+a_{k_1+1}+\cdots+a_{k_2}\geq 2$. In particular, $E_{k_2+1,k_1}^2v_1\neq 0$ and
$E_{k_2+1,k_1}^2v_{\br_0}\neq 0$.

Similar as for the infinite-dimensional case, taking $i=j=k_1$ in
\eqref{polynomial} and considering the coefficient of
$n_{k_2+1}^2n_{k_1+1}^2$, one can deduce $0\neq \sum_{\br\in
I}E_{k_2+1,k_1}^2v_{\br}\otimes \bt^{\bm+\br}\in N$ for all
$\bm\in\Z^d$. By replacing $v_{\br}$ with $E_{k_2+1,k_1}^2v_{\br}$ we
may assume that \eqref{d step-1} holds.

\textbf{Step 2.} Now we assume that \eqref{d step-1} holds. In
particular, we have
%For some $n\in\Z^d$, any $m,k$, we know that
$$\sum_{\br\in I}v_{\br}\otimes \bt^{\bn+\br}\in N\ \ \text{and}\ \ \sum_{\br\in I}v_{\br}\otimes \bt^{\bn+\bk+\br}\in N.$$
Applying $d_{\bm,i}$ and $d_{\bm-\bk,i}$ to the above elements
respectively, we get
\begin{equation*}\begin{split}
 d_{\bm,i}\sum_{\br\in I}v_{\br}\otimes \bt^{\bn+\br} =\sum_{\br\in I} \big((\bu_3 | \a+\bn+\br)+\bm\bu_3^T\big)v_{\br}\otimes \bt^{\bm+\bn+\br}\in N\\
\end{split}\end{equation*}
and
\begin{equation*}\begin{split}
 d_{\bm-\bk,i}\sum_{\br\in I}v_{\br}\otimes \bt^{\bn+\bk+\br}=\sum_{\br\in I} \big((\bu_4 | \a+\bn+\bk+\br)
 +(\bm-\bk)\bu_4^{T}\big)v_{\br}\otimes \bt^{\bm+\bn+\br}\in
 N,
\end{split}\end{equation*}
where $\bu_3=m_{i+1}\varepsilon_{i}-m_{i}\varepsilon_{i+1},
\bu_4=(m_{i+1}-k_{i+1})\e_{i}-(m_{i}-k_{i})\e_{i+1}$.
Take $\bm=2\bk$, and subtracting a multiple of $ \sum_{\br\in
I}v_{\br}\otimes \bt^{\bm+\bn+\br}$, we get
\begin{equation*}\begin{split}
\sum_{\br\in I}\big((\bu_3 | \br)+\bm\bu_3^T\big)v_{\br}\otimes \bt^{\bm+\bn+\br}\in N\\
\end{split}\end{equation*}
and
\begin{equation*}\begin{split}
  \frac{1}{4}\sum_{\br\in I}\big(2(\bu_3 | \br)+\bm\bu_3^{T}\big)v_{\br}\otimes \bt^{\bm+\bn+\br}\in N.
\end{split}\end{equation*}
These two formulas imply that
$$\sum_{\br\in I}(k_{i+1}\varepsilon_{i}-k_{i}\varepsilon_{i+1} | \br)v_{\br}\otimes t^{2\bk+\bn+\br}\in N,\ \forall\ \bk\neq0, i=1,\cdots, d-1.$$

Combining this formulas with \eqref{d step-1}, we can cancel some
term involving $v_{\br}, \br\in I$ to make $I$ smaller. Finally, we
deduce that $v_{\br}\otimes \bt^{\bn}\in N$ for some $\bn\in\Z^d$.
Repeating Step 1 to this element, we can obtain some $v\in V $ such
that $ v\otimes \bt^{\bn}\in N$ for all $\bn\in\Z^d$.
\end{proof}

\begin{proposition}\label{d:irre} Let $V$ be an irreducible
$\mathfrak{sl}_d$ and $W$ is a $\LL_d$-submodule of $F^\a(V)$. If
there exists $v\in V$ such that $v\ot\bt^\bn \in W$ for all
$\bn\in\Z$, then $W=F^\a(V)$.
\end{proposition}

\begin{proof} The proof is standard by showing that the subspace
$\{v\ |\ v\ot \bt^\bn \in W,\ \forall\ \bn\in\Z\}$ is a $\sl_d$-submodule
of $V$. We omit the details.
\end{proof}

\begin{theorem}\label{d:infinite}
Let $V$ be an irreducible $\sl_d$-module which is not isomorphic to any $V(\o_k)$ for $k=1,2,\cdots, d$, then $F^\a(V)$ is an irreducible $\LL_d$-module.
\end{theorem}

Since $\LL_d$ is a subalgebra of $\hLL$, % and $W_d$, 
we can easily deduce the following corollary.

\begin{corollary}\label{hatLL}
Let $V$ be an irreducible $\sl_d$-module which is not isomorphic to
any $V(\o_k)$ for $k=1,2,\cdots, d$, then $F^\a(V)$ is an
irreducible $\hLL_d$-module. % as well as an irreducible $W_d$-module.
\end{corollary}

Now we determine the $\LL$-submodule structure of $F^\a(V(\o_k))$
for all $k=1,\cdots,d$. First the irreducibility of $F^\a(V(\o_d))$
is clear: since $V(\o_d)=V(0)$ is the $1$-dimensional
$\sl_d$-module, we may identify $F^\a(V)=\C[t_1^{\pm1},
\cdots,t_d^{\pm1}]$ with module action $D(\bu,\br)v(\bn)=(\bu\ |\
\bn+\a)v(\bn+\br)$, which is irreducible when $\a\notin\Z^d$ and is
the direct sum of two irreducible submodules $\C \bt^{-\a}$ and
$\sum_{\a+\bn\neq 0}\C \bt^\bn$ if $\a\in\Z^d$.

Let $V_1=V(\o_1)$, the highest weight $\sl_d$-module with highest
weight $\o_1$. Then we can take $V_1=\C^d$ as a vector space such
that $\sl_d$ acts on $V$ via the matrix multiplication. Let
$\bigwedge^{k}(V_1)$ be the submodule of $k$-th tensor product module of
$V_1$ consisting of skew symmetric elements. Then $\bigwedge^{k}(V_1)$ is
just the highest weight $\sl_d$-module with highest weight $\o_k$
for any $k=1,2\cdots,d$. 

Now let $V=\bigwedge^k\C^d$ for some
$k=1,2\cdots,d$ and we consider the $\LL_d$-submodule structure of
$F^\a(V)$.
It is easy to see that for any $\a\in\C^d$, the module $F^{\a}(V)$ has a natural submodule
$$W=\{(v_1\wedge\cdots\wedge v_{k-1}\wedge(\a+\bn))\otimes \bt^{\bn}, v_1,\cdots,v_{k-1}\in \C^d\}$$ 
and for $\a\in \Z^d$， the module $F^{\a}(V)$ has additional submodules of the form
$$W'=W\oplus (V'\otimes t^{-\a}),$$
where $V'$ is an arbitrary subspace of $V$.

\begin{lemma}\label{orthg}
For any $\bm,\bn\in\Z^d, \bu\in\C^d$ with $\bn\neq0$ and
$(\bu|\bn)=0$, there exists $\bu'\in\C^d$ such that $(\bu'|\bm)=0$
and $(\bu'-x\bu|\bm-x\bn)=0$ for any $x\in\C$.
\end{lemma}

\begin{proof}
Suppose that $\bn=(n_1,\cdots,n_d)$ and $\bm=(m_1,\cdots,m_d)$.
Without loss of generality, we may assume $n_1\neq0$. Then we can
write $\bu=\sum_{i=2}^da_i(n_i\e_1-n_{1}\e_i)$ for some $a_i\in\C$.
It is easy to check that $\bu'=\sum_{i=2}^da_i(m_i\e_1-m_1\e_i)$
satisfies the requirement of the lemma.
\end{proof}

\begin{theorem}\label{d:omega_k}
Let $V=V(\omega_k)$ for $k=1,2,\cdots, d-1$. Then
\begin{itemize}
\item[(1)] If $\a\notin\Z^d$, then $F^\a(V)$ has a unique nonzero proper $\LL_d$-submodule $W$;
\item[(2)] If $\a\in\Z^d$, then any nonzero submodule of $F^\a(V)$ is of the form
$W'=W\oplus (V'\otimes t^{-\a})$, where $V'$ is an arbitrary subspace of $V$.
\end{itemize}
\end{theorem}

\begin{proof} Let $N$ be a nonzero proper submodule of $F^\a(V)$.
Choose a nonzero vector $\sum_{\br\in I}v_{\br}\ot \bt^{\br}\in
N$, where all $v_\br\in V$ are nonzero and $I\subseteq \Z^d$ is a finite subset. 

Take any $\bm_1,\bm_2\in\Z^d$ and $\bu_2\in\C^d$ with $\bm_2\neq 0, \bu_2\neq0$ 
and $(\bm_2|\bu_2)=0$.
By Lemma \ref{orthg}, we can choose $\bu_1\in\C^d$ such that $(\bm_1|\bu_1)=0$ and 
$(\bm_1-x\bm_2|\bu_1-x\bu_2)=0$ for any $x\in\C$.
Then, for $x\in\Z$, we have
\begin{equation}\label{d:m-n,n}\begin{split}
 &  D(\bu_1-x\bu_2,\bm_1-x\bm_2)D(x\bu_2,x\bm_2)(\sum_{\br\in I}v_{\br}\otimes \bt^{\br})\\
= & x\sum_{\br\in I}\big((\bu_1-x\bu_2 | \a+\br+x\bm_2)+(\bm_1-x\bm_2)(\bu_1^T-x\bu_2^T)\big)\\
  &\hskip3cm \cdot \big((\bu_2|\a+\br)+x\bm_2 \bu_2^T\big)v_{\br}\otimes \bt^{\bm_1+\br}.
\end{split}\end{equation}
By the property of the Vandermonde's determinant, we see that the
coefficients of monomials in $x$ are all in $N$. In particular, the coefficient of $x^2$ gives
\begin{equation}\label{d:poly}\begin{split}
 & \sum_{\br\in I}\Big(\big((\bu_1 |\a+\br)+\bm_1\bu_1^T\big)\bm_2 \bu_2^T\\
  &\hskip30pt +(\bu_2|\a+\br)\big((\bu_1|\bm_2)-(\bu_2 |\a+\br)-\bm_1\bu_2^T-\bm_2\bu_1^T\big)\Big) v_{\br}\otimes \bt^{\bm_1+\br}\in N.
\end{split}\end{equation}
Note that the above formula holds trivially for $\bu_2=0$.

\noindent\textbf{Claim 1.} $v_\br\ot \bt^\br\in N$ for all $\br\in
I$.

Take $\bm_1=0$ and $\bu_1=0$ in \eqref{d:poly}, then we get
\begin{equation}\label{d:m=0}
\sum_{\br\in I}(\bu_2|\a+\br)^2 v_{\br}\otimes \bt^{\br}\in N.
\end{equation}
Let $\bm_2$ vary and we see that \eqref{d:m=0}
holds for all $\bu_2\in\Z^d$. On the other hand, we also have
\begin{equation}\label{d:m}
D(\bu, \bm)\sum_{\br\in I}v_{\br}\otimes \bt^{\br}=\sum_{\br\in
I}\big((\bu|\a+\br)+\bm\bu^T\big)v_{\br}\otimes \bt^{\bm+\br}\in N,\
\forall\ (\bm|\bu)=0.
\end{equation}
Replacing $\sum_{\br\in I}v_{\br}\otimes \bt^{\br}$ with
$D(\bu,\bm)\sum_{\br\in I}v_{\br}\otimes \bt^{\br}$ in \eqref{d:m=0},
we have
\begin{equation}\label{r+m}
 \sum_{\br\in I}(\bu_2|\a+\br+\bm)^2 \big((\bu|\a+\br)+\bm\bu^T\big)v_{\br}\otimes \bt^{\bm+\br}\in N,\ \forall\ \bu_2\in\Z^d.
\end{equation}

It is easy to see that
$(\bu_2|\a+\br_1+\bm)^2=(\bu_2|\a+\br_2+\bm)^2$ as polynomials in $\bu_2$ for $\br_1\neq\br_2$ if and only if
$\a+\br_1+\bm=\pm(\a+\br_2+\bm)$, or equivalently, $2\bm=-\br_2-\br_1-2\a$. Since there are finitely many
$\br\in I$, we see that
for all but finitely many $\bm\in\Z^d$,
$(\bu_2|\a+\br+\bm)^2$ are distinct polynomials in $\bu_2$ for
distinct $\br\in I$. Thus we deduce from \eqref{r+m}, for all but finitely many
$\bm\in\Z^d$, that
\begin{equation*}
 \big((\bu|\a+\br)+\bm\bu^T\big)v_{\br}\otimes \bt^{\bm+\br}\in N,\ \forall\ \br\in I, \bu\in\C^d\ \text{with}\ (\bm|\bu)=0.
\end{equation*}
Applying $D(\bu,-\bm)$ to this element 
we obtain, for
all but finitely many $\bm\in\Z^d$, that
\begin{equation*}
 (\bu|\a+\br)^2 v_{\br}\otimes \bt^{\br}\in N,\ \forall\ \bu\in\C^d\ \text{with}\ (\bm|\bu)=0,
\end{equation*}
forcing $v_{\br}\otimes \bt^{\br}\in N$ for all $\br\in I$ with $\a+\br\neq0$. 
Moreover, if $-\a\in I$, we also have $v_{-\a}\ot \bt^{-\a}\in N$. Claim 1
follows.

So we may assume that $|I|=1$ and $v_\br=v$ in the what follows, more precisely,
we have $v\ot \bt^{\br}\in N$ for some $\br\in\Z^d$.
Replacing
$v\ot \bt^{\br}$ with $D(\bu,\bm)(v\ot \bt^{\br})$ for suitable $\bm\in\Z^d, \bu\in\C^d$ with
$(\bm|\bu)=0$ if necessary, we may assume that $\a+\br\neq0$ since $V$ is a nontrivial $\sl_d$-module.
Now we can get a basis
$\{\a+\br,\bn_1,\bn_2,\cdots,\bn_{d-1}\}$ of $\C^d$ with
$\bn_1,\bn_2,\cdots,\bn_{d-1}\in\Z^d$. 

Identify $V$ with $\bigwedge^k\C^d$, where $\C^d$ is regarded as a $\sl_d$-module via matrix multiplication.

\noindent\textbf{Claim 2.}
$\big(\bn_1\wedge\cdots\wedge\bn_{k-1}\wedge v_k\big)\ot \bt^{\br'}\in
N\setminus\{0\}$ for some $v_{k}\in\C^d$ and $\br'-\br$ is a linear
combination of $n_1,\cdots,n_{k-1}$ with coefficients $1$ or $0$;

The claim is true for $k=d$ or $k=1$. So we assume $2\leq k\leq
d-1$, which implies $d\geq 3$ in particular. Then we prove the
following assertion by induction
\begin{equation}\label{d: claim 2}
 \bn_1\wedge\cdots\wedge\bn_{l-1}\wedge v_l\ot \bt^{\br_l}\in N\setminus\{0\}
\end{equation} for some
$v_l\in\bigwedge^{k-l+1}\C^d$, $0\leq l\leq k$ such that $\br_l-\br$
is a linear combination of $\bn_1,\cdots,\bn_{l-1}$ with
coefficients $1$ or $0$. By Claim 1, we have $v\ot \bt^{\br}\in N$ for
some nonzero $v\in V$. Hence \eqref{d: claim 2} is true for $l=0$
with $v_0=v$ and $\br_0=\br$.

Now suppose \eqref{d: claim 2} holds for $l\leq k-1$. If
$\bn_l\wedge v_l=0$, then we have $v_l=\bn_l\wedge v_{l+1}$ for some
$v_{l+1}\in\bigwedge^{k-l}\C^d$, and
$\bn_1\wedge\cdots\wedge\bn_{l}\wedge v_{l+1}\ot \bt^{\br_{l+1}}\in
N\setminus\{0\}$, where $\br_{l+1}=\br_l$.

Then consider the case $\bn_l\wedge v_l\neq0$. Noticing that $l\leq
k-1\leq d-2$, we can choose a basis $\{e_i, i=1,\cdots,d\}$ of
$\C^d$ such that $e_d=\a+\br_l, e_i=\bn_i$ for $1\leq i\leq l$ and
$(e_i|e_j)=0$ for $i\neq j$ and $l+1\leq i\leq d-1$. Since
$v_l\in\bigwedge^{k-l+1}\C^d$ and $k-l+1\geq 2$, there exists
$l+1\leq i\leq d-1$ such that
$$v_l=e_i\wedge u_1+\bn_l\wedge u_2+\bn_l\wedge e_i\wedge u_3+u_4,$$
with $u_1, u_2\in\bigwedge^{k-l}(\bigoplus_{j=l+1, j\neq i}^{d}\C
e_j), u_3\in\bigwedge^{k-l-1}(\bigoplus_{j=l+1, j\neq i}^{d}\C e_j),
u_4\in\bigwedge^{k-l+1}(\bigoplus_{j=l+1, j\neq i}^{d}\C e_j)$ and
$u_1\neq 0$. Then applying $D(e_{i},\bn_l)$ on
$\bn_1\wedge\cdots\wedge\bn_{l-1}\wedge v_l\ot \bt^{\br_l}$, we get
\begin{equation*}\begin{split}
 D(e_{i},\bn_l)(\bn_1\wedge\cdots\wedge\bn_{l-1}\wedge v_l\ot \bt^{\br_l})
 = & ((\a+\br_l | e_{i})+(\bn_l e_{i}^T))\bn_1\wedge\cdots\wedge\bn_{l-1}\wedge v_l\ot \bt^{\br_l+\bn_l}\\
 = & (e_{i}^Te_i)(\bn_1\wedge\cdots\wedge\bn_{l-1}\wedge \bn_l\wedge u_1)\ot \bt^{\br_l+\bn_l}\\
 = & \bn_1\wedge\cdots\wedge\bn_{l-1}\wedge \bn_l\wedge v_{l+1}\ot \bt^{\br_{l+1}}\in
 N\setminus\{0\},
\end{split}\end{equation*}
where $v_{l+1}=u_1$ and $\br_{l+1}=\br_l+\bn_l$. Claim 2 follows.

\noindent\textbf{Claim 3.} The vector $v_{k}$ in Claim 2 lies in 
$\C\bn_1\oplus\cdots\oplus\C\bn_{k-1}\oplus\C(\a+\br')$;

Always note that $\{\a+\br', \bn_1,\cdots,\bn_{d-1}\}$ is a basis of $\C^d$. 
The claim is trivial for $k=d$, so we assume $k\le	d-1$. Without loss of generality, we assume that $v_k\in\C(\a+\br')\oplus\C\bn_{k}\oplus\cdots\C\bn_{d-1}$ 
and we need only to show $v_k\in\C(\a+\br')$. Suppose on the contrary $v_{k}\not\in\C(\a+\br')$, 
then we will deduce a contradiction.
Fix any $\bm\in\Z^d$. 

If
$\bm\in\C\bn_1\oplus\cdots\C\bn_{k-1}\oplus\C v_k$, then take any
$\bu\in\C^d$ such that $(\bu|v_k)=(\bu|\bn_i)=0$ for all
$i=1,\cdots,k-1$ and $(\bu|\a+\br')\neq0$. We have
$$D(\bu,\bm)\big((\bn_1\wedge\cdots\wedge \bn_{k-1}\wedge v_k)\ot
\bt^{\br'}\big)=(\bu|\a+\br')(\bn_1\wedge\cdots\wedge \bn_{k-1}\wedge
v_k)\ot \bt^{\br'+\bm}\in N\setminus\{0\},$$
forcing $(\bn_1\wedge\cdots\wedge \bn_{k-1}\wedge
v_k)\ot \bt^{\br'+\bm}\in N$.

If $\bm\not\in\C\bn_1\oplus\cdots\C\bn_{k-1}\oplus\C v_k$ and $k\geq
2$, we take $\bu\in\C^d$ such that $(\bu|v_k)=(\bu|\bn_i)=0$ for all $i=1,\cdots,k-1$ 
and $(\bu|\a+\br')\neq0$. Then by \eqref{d:poly},
with $\bm_1$ replaced by $\bm$, $\bm_2$ by $\bn_1$, $\bu_2$ by $\bu$, 
$\sum_{\br\in I}v_\br\ot\bt^{\br}$ by 
$(\bn_1\wedge\cdots\bn_{k-1}\wedge v_k)\ot \bt^{\br'}$ and $\bu_1$ suitably choosed, we have
\begin{equation*}\begin{split}
  & \Big(\big((\bu_1 |\a+\br')+\bm\bu_1^T\big)\bn_1\bu^T \\
  & \hskip30pt +(\bu|\a +\br')\big((\bu_1 |\bn_1)-(\bu|\a+\br')-\bm\bu^T-\bn_1\bu_1^T\big)\Big)
     (\bn_1\wedge\cdots\wedge \bn_{k-1}\wedge v_k)\ot \bt^{\br'+\bm}\\
 =&(\bu|\a+\br')\big((\bu_1 |\bn_1)-(\bu|\a+\br')-\bn_1\bu_1^T\big)\Big) (\bn_1\wedge\cdots\wedge \bn_{k-1}\wedge v_k)\ot \bt^{\br'+\bm}\\
 =&(\bu|\a+\br')^2(\bn_1\wedge\cdots\wedge \bn_{k-1}\wedge v_k)\ot \bt^{\br'+\bm}.
\end{split}\end{equation*}
Consequently, we obtain $(\bn_1\wedge\cdots\wedge \bn_{k-1}\wedge
v_k)\ot \bt^{\br'+\bm}$. %for all $\bm\in\Z^d$. Now again, Theorem \ref{d:irre} implies
%$N=F^\a(V)$, contradiction. 

If $\bm\not\in\C\bn_1\oplus\cdots\C\bn_{k-1}\oplus\C v_k$ and $k=1$.
Then $\bm\notin\C v_1$ and there exists $1\leq i_0\leq d-1$ such that
$v_1, \bm, \bn_i, i=1,\cdots,d-1, i\neq i_0$ form a basis of $\C^d$. 
For simplicity, we may assume that $i_0=d-1$, 
that is, $v_1, \bm, \bn_i, i=1,\cdots,d-2$ form a basis of $\C^d$.

For any $\bn=(n_1,\cdots,n_d)\in\Z^d$, we write $\bn=\sum_{i=1}^{d-2}x_i\bn_i+xv_1+x'\bm$ for some 
$x_i, x,x'\in\C$. 
Write $v_1=(v_{1,1},\cdots,v_{1,d})^T, \bm=(m_1,\cdots,m_d)$ and $\bn_i=(n_{i,1},\cdots,n_{i,d})$ for convenience, without loss of generality, we may assume that $v_{1,1}\neq0$. 
Take nonzero $w\in\C^d$ such that $(v_1|w)=0$ and $(\a+\br'|w)\neq0$, 
then we can write $w=\sum_{i=2}^{d}a_i(v_{1,i}\e_1-v_{1,1}\e_i)$ for some $a_i\in\C$.
Set $\bu=\sum_{i=2}^{d}a_i(n_i\e_1-n_1\e_i)$, $w'=\sum_{i=2}^{d}a_i(m_i\e_1-m_1\e_i)$, 
$w_j=\sum_{i=2}^da_i(n_{j,i}\e_1-n_{j,1}\e_i)$ for $j=1,\cdots,d-2$, and 
we see $(\bn|\bu)=(v_1|w)=(\bm|w')=(\bn_i|w_i)=0$ 
and $\bu=\sum_{i=1}^{d-2}x_iw_i+xw+x'w'$.
Note we have $(\bm-y\bn|w'-y\bu)=0$ for all $y\in\C$.

Take $\bm_1=\bm$, $\bm_2=\bn$, $\bu_2=\bu$ and $\bu_1=w'$
in \eqref{d:poly}, with $\sum_{\br\in I}v_\br\ot\bt^\br$ 
replaced by $v_1\ot\bt^{\br'}$, we have
\begin{equation*}\begin{split}
 & \left(\big((w' |\a+\br')+\bm w'\,^T\big)
    \big(\sum_{i=1}^{d-2}x_i\bn_i+xv_1+x'\bm\big)\big(\sum_{i=1}^{d-2}x_iw_i+xw+x'w'\big)^T\right.\\
  &\hskip1pt +\big(\sum_{i=1}^{d-2}x_iw_i+xw+x'w'|\a+\br'\big)
     \Big(\big(w'\big|\sum_{i=1}^{d-2}x_i\bn_i+xv_1+x'\bm\big)
            -\big(\sum_{i=1}^{d-2}x_iw_i+xw+x'w'\big|\a+\br'\big)\\
  &\hskip20pt -\left.\bm\big(\sum_{i=1}^{d-2}x_iw_i+xw+x'w'\big)^T-
    \big(\sum_{i=1}^{d-2}x_i\bn_i+xv_1+x'\bm\big) w'\,^T\Big)\right) v_1\otimes \bt^{\bm+\br'}\in N.
\end{split}\end{equation*}

Regard the element in the above formula as a polynomial in $x_1,\cdots,x_{d-2}, x'$ and the formula holds if we replace $(x_1,\cdots,x_{d-2},x')$ with any elements in $(x_1,\cdots,x_{d-2},x')+\Z^{d-1}$. 
Therefore, the coefficient of each monomial in $x_1,\cdots,x_{d-2},x'$ in the formula lies in $N$, and in particular, the term of degree $0$ with respect to $x_1,\cdots,x_{d-2},x'$ lies in $N$, that is,
\begin{equation*}\begin{split}
 & x^2\Big(\big((w' |\a+\br')+\bm w'\,^T\big)v_1w^T \\
   & \hskip20pt + \big(w|\a+\br'\big)\big((w'|v_1)-(w|\a+\br')
  -\bm w^T-v_1w'\,^T\big)\Big) v_1\otimes \bt^{\bm+\br'}\in N.
\end{split}\end{equation*}
Noticing that $(w|v_1)=0$ and we may choose $x\neq0$, the above formula is equivalent to
\begin{equation*}\begin{split}
(w|\a+\br')^2v_{1}\otimes \bt^{\bm+\br'}\in N.
\end{split}\end{equation*} 
We get $v_{1}\otimes \bt^{\bm+\br'}\in N$ in this case.

Now in all cases, we can deduce that 
$(\bn_1\wedge\cdots\wedge\bn_{k-1}\wedge v_{k})\otimes \bt^{\bm+\br'}\in N$ for all $\bm\in\Z^d$. 
Then Theorem \ref{d:irre} implies
$N=F^\a(V)$, contradiction. The claim follows.

\noindent\textbf{Remark.} By the proof of Claim 2 and Claim 3, we can also deduce the following 
two results:
\begin{itemize}
\item[(i)] If there is some nonzero $v\ot \bt^\br\in N$, then for any $\bn_1,\cdots,\bn_{k-1}$ 
such that $\a+\br,\bn_1,\cdots,\bn_{k-1}$ are linearly independent, 
we can deduce $(\bn_1\wedge\cdots\wedge\bn_{k-1}\wedge (\a+\br'))\ot \bt^{\br'}\in N$ for some 
$\br'\in\Z^d$ with $\br'-\br\in\C\bn_1\oplus\cdots\oplus\C\bn_{k-1}$; 
\item[(ii)] If $(\bn_1\wedge\cdots\wedge\bn_{k-1}\wedge v_k)\ot \bt^{\br'}\in N$ and $v_k\notin\C\bn_1\oplus\cdots\oplus\C\bn_{k-1}\oplus\C(\a+\br')$, then $(\bn_1\wedge\cdots\wedge\bn_{k-1}\wedge v_k)\ot \bt^{\bm}\in N$ for all $\bm\in\Z^d$.
\end{itemize}

\noindent\textbf{Claim 4.} $(\bn_1\wedge\cdots\wedge
\bn_{k-1}\wedge(\a+\bs))\ot \bt^\bs\in N$ for all $\bs\in \Z^d$.

By Claim 3, we have $(\bn_1\wedge\cdots\wedge
\bn_{k-1}\wedge(\a+\br'))\ot \bt^{\br'}\in N$. 
Fix any $\bm\in\Z^d$. 

First assume $k\leq d-1$. If $\bm\notin\C(\a+\br')$, we choose $\bu\in\C^d$ such that 
$(\bm|\bu)=(\bn_i|\bu)=0$ for all $i=1,\cdots,k-1$ and $(\bu|\a+\br')\neq0$, applying $D(\bu,\bm)$
we have
$$\aligned
& D(\bu,\bm)(\bn_1\wedge\cdots\wedge \bn_{k-1}\wedge(\a+\br'))\ot \bt^{\br'}\\
= & \big((\bu|\a+\br')+(\bm\bu^T)\big)(\bn_1\wedge\cdots\wedge \bn_{k-1}\wedge(\a+\br'))\ot \bt^{\br'+\bm}\\
%= & (\bu|\a+\br')(\bn_1\wedge\cdots\wedge \bn_{k-1}\wedge(\a+\br'))\ot \bt^{\br'+\bm}
%+(\a+\br'|\bu)(\bn_1\wedge\cdots\wedge \bn_{k-1}\wedge \bm)\ot \bt^{\br'+\bm}\\
= & (\bu|\a+\br')(\bn_1\wedge\cdots\wedge \bn_{k-1}\wedge(\a+\br'+\bm))\ot \bt^{\br'+\bm}\in N,
\endaligned$$
which implies $(\bn_1\wedge\cdots\wedge \bn_{k-1}\wedge(\a+\br'+\bm))\ot \bt^{\br'+\bm}\in N$.
If $\bm\in\C(\a+\br')$ but $\bm\neq -(\a+\br')$, choose $\bm'\notin\C(\a+\br')$ and $\bm-\bm'\notin\C(\a+\br'+\bm')$, then by the previous argument, 
we have $(\bn_1\wedge\cdots\wedge \bn_{k-1}\wedge(\a+\br'+\bm'))\ot \bt^{\br'+\bm'}\in N$ and 
$(\bn_1\wedge\cdots\wedge \bn_{k-1}\wedge(\a+\br'+\bm))\ot \bt^{\br'+\bm}\in N$.
If $\bm=-(\a+\br')$, the result is trivial since $\a+\br'+\bm=0$.

Now suppose $k=d$. Then $V$ is a trivial $\sl_d$-module and a similar discussion as above shows that
$(\bn_1\wedge\cdots\wedge \bn_{k-1}\wedge(\a+\br'+\bm))\ot \bt^{\br'+\bm}\in N$ for all $\bm\in\Z^d.$ 
Claim 4 is true.

By Claim 4 and the remark before it, we see that 
$$W=\sum_{\br\in\Z^d}\Big(\C(\a+\br)\wedge\bigwedge^{k-1}\C^{d}\Big)\ot\bt^{\br}\subseteq N.$$

Suppose $N\neq W$ and take any nonzero $\sum_{\br\in J}v_\br\ot\bt^{\br}\in N\setminus W$, 
where $J$ is a finite index set. By Claim 1, we see that $v_\br\ot\bt^{\br}\in N$ for all $\br\in J$.
So we may assume that $v_\br\ot\bt^\br\in N\setminus W$ and $v_\br=v_1\wedge\cdots\wedge v_{k}$.
If $\a+\br\neq0$, we must have that 
$v_1, \cdots, v_{k}, \a+\br$ is linearly independent and, by the remark following Claim 3,
we can deduce $(v_1\wedge\cdots\wedge v_{k})\ot\bt^{\br+\bm}\in N$ for all $\bm\in\Z^d$, forcing $N=F^\a(V)$, contradiction. Thus we have proved that $W$ is the only nonzero proper $\LL_d$-submodule of $F^\a(V)$ if $\a\notin\Z^d$. Assertion (1) is proved.

If $\a\in\Z^d$, the previous argument indicates that 
$v\ot\bt^{-\a}\in N\setminus W$ for some $v\in V$. Denote
$V'=\{v\in V\ |\ v\ot\bt^{-\a}\in N\}$, then we see $W'=W\oplus(V'\otimes t^{-\a})$.
%Then $D(\bu,\bm)v\ot \bt^{-\a}=(\bm\bu^T)v\ot \bt^{-\a+\bm}\in N$ for all $(\bm|\bu)=0$ shows that $V\ot\bt^{-\a}\in N$ and hence $W'\subseteq N$. 
%If $N\neq W'$, we can prove $N=F^\a(V)$ as before. 
%Thus, $W$ and $W'$ are the only nonzero proper $\LL_d$-submodule of $F^\a(V)$, 
Assertion (2) follows and the theorem is completed.
\end{proof}

\section{Modules for the $q$-analogue Algebras}

In this section we consider the similar problems for the
$q$-analogues of the algebras $\LL$ and $\hLL$. We first recall the
definitions for the corresponding algebras.

Let $q=(q_{ij})_{i,j=1}^d$ be a $d\times d$ matrix over $\C$, 
where $q_{ij}=q_{ji}^{-1}$ are roots of unity. 
We have the $d$-dimensional quantum torus $\C_q=\C_q[t_1^{\pm1},\cdots, t_d^{\pm1}]$,
which is the associative non-commutative algebra generated by
$t_1^{\pm1},\cdots,t_2^{\pm1}$ subject to the defining relations
$t_it_j=q_{ij}t_jt_i$ for $i\neq j$ and $t_it_i^{-1}=1$. As before, we
write $\bt^{\bn}=t_1^{n_1}\cdots t_d^{n_d}$ for any $\bn=(n_1,\cdots,n_d)\in\Z^d$. It is
easy to check that
$$\bt^{\bm}\bt^\bn=\sigma(\bm,\bn)\bt^{\bm+\bn},\ \ \bt^{\bm}\bt^\bn=f(\bm,\bn)\bt^{\bn}\bt^{\bm},\ \forall\ \bm,\bn\in\Z^d,$$
where $\sigma(\bm,\bn)=\prod_{1\leq i<j\leq d}q_{ji}^{m_jn_i}$ and 
$f(\bm,\bn)=\prod_{1\leq i,j\leq d}q_{ji}^{m_jn_i}$.
It is not hard to verify that
$f(\bm,\bn)=\sigma(\bm,\bn)\sigma^{-1}(\bn,\bm)$ and $f(\bm+\bn,\br)=f(\bm,\br)f(\bn,\br)$.

Denote
$$\Rad_q=\Big\{\bn\in\Z^d\ \Big|\ f(\bn,\bm)=0,\ \forall\ \bm\in\Z^d\Big\}.$$
It is well known from \cite{BGK} that the derivation Lie algebra of
$\C_q$ is
$$\Der(\C_q)=\bigoplus_{\bn\in\Z^d}\Der_n(\C_q),\ \text{where}\
\Der_n(\C_q)=\left\{\begin{array}{ll}
 \ad \bt^n, & \text{if}\ \bn\not\in\Rad_q,\\\\
 \bigoplus_{i=1}^d\C t^\bn\p_i, & \text{if}\ \bn\in\Rad_q,
\end{array}\right.$$
where $\ad \bt^\bn$ is the inner derivation with respect to
$\bt^\bn$ and $\p_i$ is the degree derivation with respect to the
variable $t_i$, that is, $\p_i(\bt^\bn)=n_i\bt^\bn$, for $i=1,\cdots,d$.
As before, we also denote $D(\bu,\bn)=\bt^{\bn}\sum_{i=1}^du_i\p_i$ for $\bu\in\C^d$ 
and $\bn\in\Rad_q$. The Lie bracket of $\Der(\C_q)$ can be given explicitly as 
\begin{equation*}\begin{split}
&[\ad \bt^{\bm},\ad \bt^\bn]=\ad [\bt^\bm,\bt^\bn],\\ %\forall\ \bm,\bn\in\Z^2\setminus\Rad_q\\
&[D(\bu,\br), \ad \bt^\bn]=(\bu|\bs)\ad \bt^{\br}\bt^{\bs},\\
&[D(\bu,\br), D(\bv,\bs)]=\sigma(\br,\bs)\big((\br|\bv)D(\bu,\br+\bs)-(\bs|\bu)(\bv,\br+\bs)\big),%\ \forall\ \br,\bs\in\Rad_q
\end{split}\end{equation*}
for all $\bm, \bn\notin\Rad_q, \br, \bs\in\Rad_q$ and $\bu,\bv\in\C^d$.

Similarly, for any $\alpha=(\a_1,\cdots,\a_d)\in\C^d$ and $\mathfrak{gl}_d$-module $V$,
the tensor space $F_q^{\a}(V)=\C_q\ot V$ admits an $\Der(\C_q)$-module
structure defined by
\begin{equation}\label{mod_LLq}\begin{array}{l}
    (\ad \bt^\bm)(\bt^{\bn}\otimes v)=[\bt^{\bm},\bt^{\bn}]\otimes v,\\\\
    D(\bu,\br)(\bt^{\bn}\otimes v)=\sigma(\br,\bn)\bt^{\br+\bn}\otimes \left((\bu|\bn+\a)+(\br\bu^T)\right)v,\\
\end{array}\end{equation}
fro any $\bm\notin\Rad_q, \br\in\Rad_q, \bn\in\Z^d, \bu\in\C^d$ and $v\in V$.

The Lie algebra $\Der(\C_q)$ has a natural subalgebra
$$\LL_d(q)=\bigoplus_{\bn\in\Z^2}\LL_d(q)_n,\ \text{where}\
\LL_d(q)_n(\C_q)=\left\{\begin{array}{ll}
 \ad \bt^n&\text{if}\ \bn\not\in\Rad_q,\\\\
 \sum_{i,j=1}^d\C t^\bn(n_j\p_i-n_i\p_j)&\text{if}\ \bn\in\Rad_q.
\end{array}\right.$$
Adding the degree operators, we get another subalgebra
$\hLL_d(q)=\LL_d(q)\oplus\sum_{i=1}^d\C\p_i$, which is called the skew
derivation Lie algebra of $\C_q$. We see that $\hLL_d(q)$ is spanned by all
$\ad \bt^\bn$ for $\bn\not\in\Rad_q$ and $D(\bu,\bn)$ for $\bu\in\C^d, \bn\in\Rad_q$ with $(\bu|\bn)=0$.
%Denote $d_\br=t^\br(r_2\p_1-r_1\p_2)$ for $\br\in\Rad_q$, then the
%Lie bracket of $\LL(q)$ is given by
%\begin{equation*}\begin{split}
%&[\ad \bt^{\bm},\ad \bt^\bn]=\ad [\bt^\bm,\bt^\bn],\\ %\forall\ \bm,\bn\in\Z^2\setminus\Rad_q\\
%&[d_{\br}, \ad \bt^\bn]=\det{\bn\choose\br}\ad \bt^{\bn+\br},\\
%&[d_{\br}, d_\bs]=\det{\bs\choose\br}d_{\br+\bs},%\ \forall\ \br,\bs\in\Rad_q
%\end{split}\ \ \forall\ \bm, \bn\notin\Rad_q, \br, \bs\in\Rad_q.\end{equation*}
%
Now we can regard $F_q^\a(V)$ as an $\LL_d(q)$-module as well as an $\hLL_d(q)$-module.

In what follows, we will denote $\LL(q)=\LL_d(q), \hLL(q)=\hLL_d(q)$ and $\LL=\LL_d, \hLL=\hLL_d$ for simplicity.
%
%-module by defining
%the action of the degree operators
%\begin{equation}\label{mod_hLLq}
%    \p_i(\bt^\bn\ot v)=(\a_i+n_i)(\bt^\bn\ot v),\ \forall\ \bn\in\Z^2, v\in
%    V, i=1,\cdots,d.
%\end{equation}

Remark that when $q_{i,j}=1$ for all $i,j=1,\cdots,d$, we have
$\C_q[t_1^{\pm1},\cdots,t_d^{\pm1}]=\C[t_1^{\pm1},\cdots,t_d^{\pm1}]$ is just the
usual Laurent polynomial algebra and $\LL(q)$ and $\hLL(q)$ are just the algebras
$\LL$ and $\hLL$ we studied in the previous sections. 

Note that $[\C_q,\C_q]=\sum_{\bn\notin\Rad_q} \C\bt^{\bn}$. We see
that $\LL'(q)=\sum_{\bn\in\Rad_q}  \sum_{i,j=1}^d\C t^\bn(n_j\p_i-n_i\p_j)$ and 
$\hLL'(q)=\LL'(q)\oplus\sum_{i=1}^d\C\p_i$ are subalgebras of
$\LL(q)$ and $\hLL(q)$, respectively. Further more, $\ad
[\C_q,\C_q]=\sum_{\bm\notin\Rad_q}\C\ad \bt^{\bm}$ is an ideal of
both $\LL(q)$ and $\hLL(q)$. We have
\begin{equation}\label{iso_LL}
\LL(q)/\ad [\C_q,\C_q]\cong\LL'(q)\cong \LL\quad\text{and}\quad \hLL(q)/\ad
[\C_q,\C_q]\cong\hLL'(q)\cong \hLL.
\end{equation} 

By Theorem III.2 oin \cite{N} (or cf. \cite{LZ2} Lemma 3.3), by replacing $t_1,\cdots, t_d$ with another set of suitable generators $\bt^{\bn_1},\cdots,\bt^{\bn_d}$, 
we can assume that $q_{2i-1,2i}=q_{2i,2i-1}^{-1}$ is a primitive root of unity of order $l_{2i-1}=l_{2i}\geq 2$ 
for any $i=1,\cdots,d_0$ with $2d_0\leq d$ and all other $q_{ij}=1$. 
For convenience, denote $l_i=1$ for $2d_0+1\leq i\leq d$ and $\bl=(l_1,\cdots,l_d)$. 
Fix these notations from now on. Under this assumption, we have $\Rad_q=\bigoplus_{i=1}^dl_i\Z$.
Then the Lie algebra isomorphisms
in \eqref{iso_LL} can be given by 
\begin{equation}\label{iso'_LL}
D(\bu,\bn)\mapsto D\Big(\sum_{i=1}^dl_iu_i, \sum_{i=1}^d\frac{n_i}{l_i}\Big),
\end{equation}
for all $\bu=(u_1,\cdots,u_d)\in\C^d$ and $\bn=(n_1,\cdots,n_d)\in\Rad_q$ with $(\bu|\bn)=0$.

%$$D(\sum_{i=1}^d\frac{u_i}{l_i},\sum_{i=1}^dl_in_i)\mapsto D(\bu,\bn)$$
%
%{\color{red} where the isomorphism
%is given by $d_{\bm}\mapsto l^2d_{\bm/l}, \p_i\mapsto l\p_i, i=1,2$.}

%\begin{remark} The above modules for $\hLL(q)$ were first defined in \cite{LT},
%where these modules are parameterized by some $\a\in\C$ and some
%function $g:\Z^d\rightarrow \C^*$ satisfying certain conditions. In
%\cite{LZ2}, it was shown that any modules define in \cite{LT} is
%isomorphic to one of the modules defined by \eqref{mod_LLq} and
%\eqref{mod_hLLq}.
%\end{remark}

We first consider $F_q^\a(V)$ as a module over the subalgebras
$\LL'(q)\cong\LL$ or $\hLL'(q)\cong\hLL$. Denote $\GG=\LL\ \text{or}\
\hLL$, $\GG'(q)=\LL'(q)\ \text{or}\ \hLL(q)'$ for convenience. 
Set %$\bl=(l_1,l_2,\cdots,l_d)$ and
$$I=\{\bi=(i_1,\cdots,i_d)\in\Z^d\ |\ 0\leq i_j< l_j\},$$ 
then we have the $\GG'(q)$-module decomposition of $F_q^\a(V)$ as follows
$$F_q^\a(V)=\sum_{\bi\in I}F_q^{\a, \bi}(V),\ 
\text{where}\ F_q^{\a,\bi}(V)=\sum_{\br\in\Rad_q}\C \bt^{\br+\bi}\ot V.$$ 

For the $\gl_d$-module $V$, we define a new action of $\gl_d$ on $V$ as follows:
for any $B\in\gl_d$ and $v\in V$, we set $B\circ v=LBL^{-1}v$, 
where $L=(\delta_{ij}l_i)_{i,j=1}^d$ is the an invertible matrix.
Denote this new $\gl_d$-module by $V^{(\bl)}$.
Using the Lie algebra isomorphisms \eqref{iso_LL} given by \eqref{iso'_LL}, we can view
each $F_q^{\a,\bi}(V)$ as a $\GG$-module and moreover, we have the $\GG$-module isomorphism
\begin{equation}\label{iso_mod}
%\begin{array}{lll}
F_q^{\a,\bi}(V) \longrightarrow F^{\a_{\bi}}(V^{(\bl)}), \bt^{\bn+\bi}\ot v \mapsto t_1^{n_1/l_1}\cdots t_d^{n_d/l_d}\ot v,\ \forall\ \bn=(n_1,\cdots,n_d)\in \Rad_q,
%\end{array}
\end{equation} 
where $\bi=(i_1,\cdots,i_d)\in I$, $\a_{\bi}=(\frac{\a_1+i_1}{l_1},\cdots,\frac{\a_d+i_d}{l_d})$ and $F^{\a_{\bi}}(V^{(\bl)})$ is the $\GG$-module studied in Section 3.

Noticing that $(\ad [\C_q,\C_q])F_q^{\a,\bf{0}}=0$ by \eqref{mod_LLq}, we see that the
$\LL(q)$ or $\hLL(q)$-module structure of $F_q^{\a,\bf{0}}(V)$ is completely
determined by the $\GG'(q)\cong\GG$-module structure of $F^{\a_{\bf{0}}}(V)$, which is completely determined
by results in Section 3. Now, as in \cite{LZ2}, we denote
$$G_q^\a(V)=\sum_{\bi\in I, \bi\neq \bf{0}}F_q^{\a,\bi}(V).$$ 
Then we have
%
%It is easy to check that $F_q^\a(V)$ can be decomposed as the direct
%sum of two submodules
%$$F_q^\a(V)=\big(Z\ot V\big)\oplus\big(C'_q\ot V\big).$$%, which is isomorphic to $[\C_q,\C_q]$.
%Moreover, we have that $(\ad \bt^\bm)(Z\ot V)=0$ for all $\bm\notin
%\Rad_q$.
%
% Since the module
%structure of $F^{\a/p}(V)$ are already determined in Section 3, it
%suffices to consider the module structure of $\C'_q\ot V$, which we
%denote by $G^\a(V)$ as in \cite{LZ2}.

\begin{theorem}
$G_q^\a(V)$ is an irreducible $\LL(q)$-module and hence an irreducible
$\hLL(q)$-module.
\end{theorem}

\begin{proof} Take any nonzero $\LL(q)$-submodule $N$ of $G^\a_q(V)$, then
$N_{\bi}=N\cap F_q^{\a,\bi}(V)$ is an $\LL'(q)$-submodule of
$F_q^{\a,\bi}(V)$ for any $\bi\in I\setminus\{0\}$.

\noindent\textbf{Claim 1.} $N_{\bi}\neq0$ for some $\bi\in
I\setminus\{0\}$.

\def\F{\mathbb{F}}

Take any nonzero $w=\sum_{\bi\in J}\bt^{\br+\bi}\otimes v_{\bi}\in
N$ for some $\br=(r_1,\cdots,r_d)\notin \Rad_q$, some finite subset
$J\subseteq\Z^d$ with $0\in J$, and nonzero $v_\bi\in V, \bi\in J$.
If $J\subseteq\Rad_q$, the claim already holds.

Now suppose $J\not\subseteq\Rad_q$ and take any $\bj=(j_1,\cdots,j_d)\in
J\setminus\Rad_q$. Recall $\Rad_q=\bigoplus_{i=1}^dl_i\Z$, we may assume that not both of
$j_1$ and $j_2$ are divisible by $l_1=l_2$ without loss of generality. 
There exist a prime $p$ and an integer $k\in\Z_+$
such that
$$(j_1,j_2)\in (p^k\Z\times p^k\Z)\setminus(p^{k+1}\Z\times p^{k+1}\Z),\ \ l_1=l_2\in p^{k+1}\Z.$$
%$$\bj=p^k\bj',\ \ \bj'\in \Z^2\setminus(p\Z)^2,\ \ l\in p^{k+1}\Z.$$
It is obvious $(r_1,r_2)\notin(p^{k+1}\Z)^2$ or $(r_1+j_1,r_2+j_2)\notin (p^{k+1}\Z)^2$, 
and without loss of generality, we may assume that
$(r_1,r_2)\notin (p^{k+1}\Z)^2$. So there exists $\bn=(n_1,n_2,0,\cdots,0)\in\Z^d$ such that
neither $r_1n_2-r_2n_1$ nor $j_1n_2-j_2n_1$ is divisible by $p^{k+1}$. Hence neither $r_1n_2-r_2n_1$ nor $j_1n_2-j_2n_1$ is divisible by $l_1=l_2$.
%There exists $k'\in\Z_+$ such that
%$$\br=p^{k'}\br',\ \ \br'\in \Z^2\setminus(p\Z)^2,\ \ k'\leq k.$$
%For any $\bn=(n_1,n_2)\in\Z^2$, denote
%$\bar{\bn}=(\bar{n}_1,\bar{n}_2)\in(\Z/p\Z)^2$, where $\bar{n}_i$ is
%the image of $n_i$ in $\Z/p\Z$ for $i=1,2$.
%
%Now we have $\bar{\bj}'$ and $\bar{\br}'$ are both nonzero in
%$(\Z/p\Z)^2$, which is a $2$-dimensional vector space over the field
%$\Z/p\Z$, so there exists $\bn\in\Z^2$ such that
%$\det{\bar{\bj}'\choose\bar{\bn}}$ and
%$\det{\bar{\br}'\choose\bar{\bn}}$ are both nonzero in $\Z/p\Z$.
%Hence neither of $\det{\bj'\choose\bn}$ and $\det{\br'\choose\bn}$
%is divisible by $p$, or equivalently, 
%$\det{\bj\choose\bn}$ is not
%divisible by $p^{k+1}$ and $\det{\br\choose\bn}$ is not divisible by
%$p^{k'+1}$. In particular, neither of $\det{\br\choose\bn}$ and
%$\det{\bj\choose\bn}$ is divisible by $l$. 
%%, i.e., $[\bt^\bn, \bt^{\bj}]\neq0$ and $[\bt^{\bn+\bj}, \bt^{\bn+\bj'}]\neq0$.

Set $\br'=(r_1,r_2,0,\cdots,0)$ and $\bj'=(j_1,j_2,0,\cdots,0)\in\Z^d$. 
If $q_{21}^{j_1r_2}-q_{21}^{j_2r_1}\neq0$, we have
$$0\neq w'=(\ad \bt^{\br'})w=\sum_{\bi\in J}[\bt^{\br'}, \bt^{\br+\bi}]\otimes v_{\bi}=\sum_{\bi\in J\setminus\{0\}}q_{21}^{r_1r_2}\big(q_{21}^{i_1r_2}-q_{21}^{i_2r_1}\big)\bt^{\br+\bi+\br'}\otimes v_{\bi}\in N.$$
Now suppose $q_{21}^{j_1r_2}-q_{21}^{j_2r_1}=0$, or equivalently, $j_1r_2-j_2r_1$
is divisible by $l_1=l_2$. First applying $\ad \bt^{\bn}$ on $w$ we get
$$0\neq (\ad \bt^{\bn})w=\sum_{\bi\in J}[\bt^{\bn}, \bt^{\br+\bi}]\otimes v_{\bi}=\sum_{\bi\in J}\bt^{\bn+\br+\bi}\otimes v'_{\bi}\in N,$$
where
$v'_{\bi}=\big(q_{21}^{(r_1+i_1)n_2}-q_{21}^{(r_2+i_2)n_1}\big)v_{\bi}\in V$
with $v'_0\neq0$ since $r_2n_1-r_1n_2$ is not divisible by
$l_1$. Next we apply $\ad \bt^{\bn+\br'+\bj'}$ and obtain
$$w'=(\ad \bt^{\bn+\br'+\bj'})(\ad \bt^{\bn})w=(\ad \bt^{\bn+\br'+\bj'})\sum_{\bi\in
J}\bt^{\bn+\br+\bi}\otimes v'_{\bi}=\sum_{\bi\in
J\setminus\{\bj\}}[\bt^{\bn+\br'+\bj'}, \bt^{\bn+\br+\bi}]\otimes
v'_{\bi}\in N.$$ 
Since $n_1j_2-n_2j_1$ is not divisible by $l_1=l_2$, we have
$$[\bt^{\bn+\br'+\bj'}, \bt^{\bn+\br}]=q_{21}^{(n_1+r_1)(n_2+r_2)}(q_{21}^{(n_1+r_1)j_2}-q_{21}^{(n_2+r_2)j_1})=q_{21}^{(n_1+r_1)(n_2+r_2)}q_{21}^{j_1r_2}(q_{21}^{n_1j_2}-q_{21}^{n_2j_1})\neq0$$
and hence $0\neq w'\in N$.

Replacing $w$ with $w'$, we have made the index set $J$ smaller.
Repeat this process several times, we can reach a nonzero element in
$N_{\bi}$ for some $\bi\in I\setminus\{0\}$ and the claim follows.

%We have $j'_1j_2-j'_2j_1$ is divisible by $p$ for all $\bj, \bj'\in
%J$.
%
%Now suppose that $\bj-\bj'\notin\Rad_q$ for some $\bj=(j_1,
%j_2),\bj'=(j'_1,j'_2)\in J$. Without loss of generality, we assume
%that $j_1-j'_1$ is not divisible by $p$. Now by \eqref{mod_LLq}, we
%get
%$$(\ad \bt^{(0,1)})(w)=\sum_{\bi\in J}[\bt^{(0,1)}, \bt^\bi]\otimes v_{\bi}=\sum_{\bi\in J}q^{i_1}\bt^{\bi+(0,1)}\otimes v_{\bi}\in N.$$
%Hence $$(\ad \bt^{(0,1)})(w)-q^{j_1-j'_1}w=\sum_{\bi\in
%J}q^{i_1}\bt^{\bi+(0,1)}\otimes v_{\bi}\in N$$

By the above claim, we can take $\bi\in I\setminus\{0\}$ such that
$N_{\bi}\neq 0$. Since $\LL'(q)\cong \LL$, we have that $N_{\bi}$ is
isomorphism to one of the $\LL$-submodules of $F^{\a_{\bi}}(V^{(\bl)})$ (cf. \eqref{iso_mod}) as
described in Section 3. In particular, there exist
$\br\in\Rad_q$ and nonzero $v\in V$ such that $\bt^{\br+\bi}\ot
v\in N_{\bi}$ by Theorem \ref{d:infinite} and Theorem \ref{d:omega_k}. 
Fix this $v$, and by the module action
\eqref{mod_LLq}, we see that
$$K=\{\bt^{\bn}\ |\ \bn\not\in\Rad_q, \bt^\bn\ot v\in N\}$$
is a nonzero $\Z^d$-graded ideal of the Lie algebra $[\C_q,\C_q]$.
By Lemma 2.2 of \cite{LZ2}, $[\C_q,\C_q]$ is a $\Z^d$-graded simple
Lie algebra, so we have $K=[\C_q,\C_q]$ and hence $[\C_q,\C_q]\ot
v\subseteq N$. Now let
$$V'=\{v'\in V\ |\ [\C_q,\C_q]\ot v'\in N\}.$$
Again by the module action \eqref{mod_LLq} we see that $V'$ is
stable under the action of $\br^T\bu$ for all $\br\in\Rad_q, \bu\in\C^d$ with $(\bu|\br)=0$, hence under
the action of $\sl_d$. We obtain that $V'$ is a nonzero submodule of
$\sl_d$ and hence $V'=V$. The theorem is completed.
%
%$N_{i,j}=F^{\a,(i,j)}$ if $\dim(V)\geq 3$ by Lemma \ref{element},
%$N_{i,j}=\sum_{\br\in\Rad_q}\bt^{\br+{(i,j)}}\ot \C(\br+\a)$ if
%$V=\C^2$ is the natural module of $\sl_2$ by Lemma \ref{dim=2}, and
%$N_{i,j}=\C t^{(i,j)}\ot V$ or
%$N_{i,j}=\sum_{\br\in\Rad_q}\bt^{\br+{(i,j)}}\ot \C(\br+\a)$ if
%$V=\C$ is the $1$-dimensional trivial $\sl_2$-module.
\end{proof}

%\begin{remark}
%The proof of Claim 1 in the above theorem is also valid for the case
%$q$ is not a root of unity and $\Rad_q=0$, and this yields that
%$G^\a_q(V)=\sum_{\bn\neq0}\C\bt^{\bn}\otimes V$ is an irreducible
%$\LL(q)$-module as well as an irreducible $\hLL(q)$-module when $q$
%is not a root of unity. On the other hand, when we take $V$ as the
%$1$-dimensional trivial $\sl_2$-module, we see that $G^\a_q(V)\cong
%\ad [\C_q,\C_q]$ is an irreducible $\ad [\C_q,\C_q]$-module, in
%other words, $[\C_q,\C_q]=\sum_{\bn\notin\Rad_q}\ad \bt^{\bn}$ is a
%simple Lie algebra. Equivalently, $[\C_q,\C_q]$ is a simple Lie
%algebra.
%\end{remark}

\begin{equation}\label{mod_LLq}\begin{array}{l}
    (\ad \bt^\bm)(\bt^{\bn}\otimes v)=[\bt^{\bm},\bt^{\bn}]\otimes v,\\\\
    D(\bu,\br)(\bt^{\bn}\otimes v)=\sigma(\br,\bn)\bt^{\br+\bn}\otimes \left((\bu|\bn+\a)+(\br\bu^T)\right)v,\\
\end{array}\end{equation}
fro any $\bm\notin\Rad_q, \br\in\Rad_q, \bn\in\Z^d, \bu\in\C^d$ and $v\in V$.

Summarize the results in this section, we can conclude that

\begin{theorem}\label{main_q} Let $V$ be an irreducible
$\mathfrak{gl}_d$-module (maybe infinite-dimensional) and
$\alpha=(\a_1,\cdots,\a_d)\in\C^d$. Denote $\GG=\LL\ \text{or}\ \hLL$,
$\GG(q)=\LL(q)\ \text{or}\ \hLL(q)$. The $\GG(q)$-module $F^\a_q(V)$
can be decomposed as the direct sum of two submodules
$$F^\a_q(V)=F_q^{\a,\bf{0}}(V)\oplus G^\a_q(V).$$ Moreover,
$\big(\ad[\C_q,\C_q]\big)F^{\a,\bf{0}}(V)=0$ and as a
$\GG(q)/\ad[\C_q,\C_q]\cong\GG$-module $F^{\a,\bf{0}}(V)$ is
isomorphic to the $\GG$-module $F^{\a_{\bf{0}}}(V^{(\bl)})$ 
via \eqref{iso_LL} and \eqref{iso_mod}, whose module structure is completely determined in Section 3; 
and $G^\a_q(V)$ is an irreducible $\GG(q)$-module.
%\begin{itemize}
%\item[(1)] If $\dim(V)\geq 3$, then $F^\a(V)$ is an irreducible $\GG$-module;
%\item[(2)] If $\dim(V)= 1, \a\notin\Z^2$, then $F^\a(V)$ is an irreducible $\GG$-module;
%\item[(3)] If $\dim(V)= 1, \a\in\Z^2$, then $F^\a(V)$ is the sum of two irreducible $\GG$-submodules:
%        $$\big(\C t^{-\a}\ot V\big)\bigoplus \Big(\sum_{\bi\in\Z^2, \bi\neq -\a}\C t^{\bi}\ot V\Big);$$
%\item[(4)] If $\dim(V)= 2, \a\notin\Z^2$, then, identifying $V$ as the natural module $\C^2$ of
%$\sl_2$, we have $F^\a(V)$ has a unique proper submodule
%$W=\sum_{\bn\in\Z^2}\C t^\bn\otimes(\a+\bn)$;
%\item[(5)] If $\dim(V)= 2, \a\notin\Z^2$, then, identifying $V$ as the natural module $\C^2$, we have that any $\GG$-submodule of $F^\a(V)$ is of the
%form $W\oplus(t^{-\a}\ot V')$ for some subspace $V'\subseteq V$.
%\end{itemize}
\end{theorem}

%\begin{remark} The result for $\hLL(q)$ and finite-dimensional $V$ recovers the main results of
%\cite{LT}. The results for $\LL(q)$ or for $\hLL(q)$ with infinite
%dimensional $V$ are new.
%\end{remark}

\noindent {\bf Acknowledgments.} This work is partially supported by
the NSF of China (Grant 11471294) and the Foundation for Young
Teachers of Zhengzhou University (Grant 1421315071). 
%We would like
%to thank the referee for pointing out the inaccuracies and providing
%valuable suggestions.


\begin{thebibliography}{99999}

\bibitem[AABGP]{AABGP} B.N. Allison, S. Azam, S. Berman, Y. Gao and A. Pianzola, 
Extended affine Lie algebras and their root systems, {\it Mem. Amer. Math. Soc.}, 126 (605) (1997).

\bibitem[B]{B}Y. Billing, Jet modules, {\it Canad. J. Math.}, 59(2007), 712--729.

\bibitem[BGK]{BGK} S. Berman, Y. Gao and Y. Krylyuk, Quantum tori and the structure of elliptic quasi-simple Lie algebras. {\it J. Funct. Anal.}, 135(1996), no. 2, 339--389.  

\bibitem[BF]{BF} Y. Billing and V. Futorny, Classification of irreducible representations of Lie algebra of vector fields on a torus. {\it J. Reine Angew. Math.}, 720 (2016), 199--216.

\bibitem[BMZ]{BMZ} Y. Billing, A. Molev and R. Zhang, Differential equations in vertex algebras and simple modules for the Lie algebra of vector fields on a torus. {\it Adv. Math.}, 218 (2008), 1972--2004.

\bibitem[BT]{BT} Y. Billig and J. Talboom, Classification of Category $\mathcal{J}$ Modules for Divergence Zero Vector Fields on a Torus. Preprint, arXiv:1607.07067.

\bibitem[DZ]{DZ} D. Djokovic and K. Zhao, Some infinite-dass of Lie algebras of Block timensional simple Lie algebras in characteristic 0 related to those of Block, {\it J. Pure Appl. Algebra}, 127(1998), no.2, 153--165.

\bibitem[E1]{E1} S. Eswara Rao, Irreducible representations of the Lie algebra of the diffeomorphisms of a $d$-dimensional torus, {\it J. Algebea}, 182 (1996), no 2, 401--421.

\bibitem[E2]{E2} S. Eswara Rao, Partial classification of modules for Lie-algebra of diffeomorphisms of 
$d$-dimensional torus, {\it J. Math. Phys.}, 45(8), (2004) 3322--3333.

%\bibitem[GL]{GL} X. Guo and G. Liu, Jet modules for the Virasoro-like algebra. Preprint. ArXiv:1611.02034.

%\bibitem[GLi]{GLi} X. Guo and X. Liu. Whittaker modules over Virasoro-like algebras, {\it J. Math. Phy.}, 52(2011), 093504.

\bibitem[GLZ]{GLZ} X. Guo, G. Liu and K. Zhao, Irreducible Harish-Chandra modules over extended Witt algebras, {\it Ark. Mat.}, 52 (2014), 99--112.

\bibitem[GZ]{GZ} X. Guo and K. Zhao, Irreducible weight modules over Witt algebras, {\it Proc. Amer. Math. Soc.}, 139(2011), 2367--2373.

%\bibitem[JL]{JL} J. Jiang and W. Lin, Partial classification of cuspidal simple modules for Virasoro-like algebra. {\it J. Algebra.}, 464 (2016), 266--278.

\bibitem[L1]{L1} T. A. Larsson, Multi dimensional Virasoro algebra, {\it Phys. Lett.}, B 231, 94--96(1989).

\bibitem[L2]{L2} T. A. Larsson, Central and non-central extrension of multi-graded Lie algebras, {\it J. Phys., A}, 25, 6493--6508(1992).

\bibitem[L3]{L3} T. A. Larsson, Conformal fields: A class of representation of Vect(N), 
{\it Int. J. Mod. Phys. A}, 7, 6493--6508(1992).

\bibitem[LT1]{LT1} W. Lin and S. Tan, Representations of the Lie algebra of derivations for quantum torus, 
{\it J. Algebra}, 275(2004), 250--274. 

\bibitem[LT2]{LT2} W. Lin and S. Tan, The  representation of the skew derivation Lie algebras over the quantum torus, {\it Adv. Math. (China)}, 34(4) (2005)477--487.

\bibitem[LZ1]{LZ} G. Liu and K. Zhao, New irreducible weight modules over Witt algebras with infinite-dimensional weight spaces. {\it Bull. Lond. Math. Soc.}, 47 (2015), no. 5, 789--795.

\bibitem[LZ2]{LZ1.5} G. Liu and K. Zhao, Irreducible Harish-Chandra modules over the derivation algebras of rational quantum tori. {\it Glasg. Math. J.}, 55 (2013), no. 3, 677--693.

\bibitem[LZ3]{LZ2} G. Liu and K. Zhao, Irreducible modules over the derivation algebras of rational quantum tori. {\it J. Algebra}, 340 (2011), 28--34.

%\bibitem[M]{M} O. Mathien; Classification of irreducible weigh modules. {\it Ann. Inst. Fourier (Grenoble)}, (2000), no. 2,537--592.

\bibitem[MZ]{MZ} V. Marzuchuk, K. Zhao, Supports of weight modules over Witt algebras, 
{\it Proc. Roy. Soc. Edinburgh Sect. A}, 141(2011), 155--170.

\bibitem[N]{N} K. Neeb, the classification of rational quantum tori and the structure of their automorphism groups. {\it Canad. Math. Bull.}, 51 (2008), no. 2, 261--282.

\bibitem[Sh]{Sh} G. Shen, Graded modules of graded Lie algebra of Cartan type. I. Mixed products of modules,  {\it Sci. Sinica Ser.}, A 29(1986), no.6, 570--581.

\bibitem[T]{T} J. Talboom, Irreducible modules for the Lie algebra of divergence zero vector fields on a torus. (English summary). {\it Comm. Algebra}, 44 (2016), no. 4, 1795--1808. 

\bibitem[Z]{Z} K. Zhao, Weight modules over generalized Witt algebras with 1-dimensional weight spaces, 
{\it Forum Math.}, 16(2004), no. 5, 725--748.

%\bibitem[KR]{KR} V. Kac and A. Raina, Bombay lectures on Highest Weight Representations of Infinite Dimensional Lie Algebra, World Scientific, Singapore, 1987.

\end{thebibliography}
\end{document}